\documentclass[10pt]{amsart}
\usepackage{amsmath, amssymb, amsthm, latexsym, amscd, enumerate, MnSymbol}
\usepackage{nicefrac,xspace,tikz}
\usepackage[all]{xy}
\usepackage{graphicx}
\usepackage[vcentermath]{youngtab}
\usepackage{nccmath}
 \newlength{\baseunit}               
 \newcount{\numlines}                
\newtheorem{theorem}{Theorem}

\newtheorem{lemma}[theorem]{Lemma}
\newtheorem{remark}[theorem]{Remark}
\newtheorem{prop}[theorem]{Proposition}

\newcommand{\cO}{{\mathcal O}}

\newcommand{\mZ}{\mathbb{Z}}

\newcommand{\la}{\lambda}

\newcommand{\op}{\operatorname}

\DeclareMathOperator{\Hom}{Hom}   
\DeclareMathOperator{\End}{End}   
   \DeclareMathOperator{\Mod}{mod}
   
\DeclareMathOperator{\gmod}{gmod}

\begin{document}
\pagestyle{plain}
\title[draft]{Categorified Jones-Wenzl Projectors:\\ a comparison}
\author{Catharina Stroppel}
\author{Joshua Sussan}
\maketitle

\begin{abstract}
We explicitly describe a relationship between the Lie theoretic and topological categorification of the Jones-Wenzl projector.
\end{abstract}

\section{Introduction}
The Jones-Wenzl projector $p_n$ is an endomorphism of the $n$-fold tensor product $V^{\otimes n}$ of the natural representation of the quantum group $U_q(\mathfrak{sl}_2)$ for $\mathfrak{sl}_2$. It appears as a standard tool for decomposing explicitly the finite dimensional representations and plays an important role in the definition and construction of the corresponding Reshetikhin-Turaev $3$-manifold invariant. A categorification of these projectors seems therefore a necessary tool in establishing a categorification of the $3$-manifold invariant.

A Lie theoretic categorification of the Jones-Wenzl projector $p_n$ was constructed in ~\cite{FSS1}, based on ~\cite{FKS},
using a graded version ${}^{\mathbb{Z}} \mathcal{O}(\mathfrak{gl}_n)$ of the category $\cO$ of highest weight modules for $ \mathfrak{gl}_n$. The direct sum $\bigoplus_{i=0}^n({}^{\mathbb{Z}} \mathcal{O}_i(\mathfrak{gl}_n))$
of maximal singular blocks
is a category whose Grothendieck group $V^{\otimes n}$ is naturally isomorphic as $\mZ[q,q^{-1}]$-module to $V^{\otimes n}$.
A major problem in categorifying the projection lies in the fact that it is only defined over the rational numbers. However,
the Jones-Wenzl projector can be defined integrally when passing to the completion of $V^{\otimes n}$. The resulting $\mZ[q^{-1}][[q]]$-module then  got realized in \cite{FSS1} as a certain completed Grothendieck group corresponding to certain subcategory of the one-sided unbounded derived category $\oplus_{i=0}^n D^{<}({}^{\mathbb{Z}} \mathcal{O}_i(\mathfrak{gl}_n))$.

A more topological categorification of the Jones-Wenzl projector was found by Cooper and Kruskhal ~\cite{CoKr} using techniques of Bar-Natan ~\cite{BN}. It is based on Khovanov's arc algebra from \cite{KhovJones}. The latter is known to be directly related to the highest weight categories for $\mathfrak{gl}_n$ mentioned above by \cite{StrSpringer} and \cite{BS3} via an equivalence of abelian categories composed with a Koszul duality. As stated already in \cite{FSS1}, the two categorifications of the Jones-Wenzl projector are therefore supposed to be related by an equivalence of categories composed with Koszul duality. In this note we verify this conjecture for the case where an explicit description of the Cooper-Krushkal categorified projector is available. We explicitly compare the Lie theoretic with the topological construction. For general $n$, the statement should follow from the (abstract) characterization of the categorified Jones-Wenzl projector, but (apart from $n=3$) it would need much more effort (if possible at all) to make the relationship and constructions explicit.

Finally we want to mention yet another categorification of the Jones-Wenzl projector due to Rozansky, \cite{Roz} which is quite different in nature. It arises as a limit of a certain direct system of complexes. Lie theoretically this corresponds to taking limits of so-called shuffling functors (and their Koszul duals, \cite[Section 6.5]{MOS}) which we will not work out in detail here.

We first recall in Section ~\ref{lietheory} the abstract Lie theoretic construction which should give already a rough idea of the general categorification.  After that we specialize to the case $n=2$ where we work out the setup in detail and make the categories and functors explicit.  We finally prove in Section ~\ref{dualitytheoremsection} the main result saying that the Lie theoretic construction of the categorified Jones-Wenzl projector from ~\cite{FSS1} is related to the topological construction from ~\cite{CoKr} by Koszul duality.

\subsection*{Acknowledgements}
The first author is deeply grateful to Olaf Schn\"urer for several interesting discussions. The second author is grateful for the support from the Max Planck Institute for Mathematics in Bonn and its excellent working conditions.

\section{A Lie theoretic categorification}
\label{lietheory}
Let $ \mathfrak{gl}_n = \mathfrak{n}^- \oplus \mathfrak{h} \oplus \mathfrak{n}^+ $ be the triangular decomposition of the Lie algebra of complex $n\times n$ matrices into the direct sum of strictly lower triangular, diagonal, and strictly upper triangular matrices respectively.  Set $\mathfrak{b} = \mathfrak{h} \oplus \mathfrak{n}^+$.
Define $ \mathcal{O}(\mathfrak{gl}_n) $ to be the full subcategory of $ \mathcal{U}(\mathfrak{gl}_n) $-modules which are finitely generated, diagonalizable with respect to $ \mathfrak{h} $, and locally finite with respect to
$ \mathcal{U}(\mathfrak{b})$.
The dual space $ \mathfrak{h}^* $ has basis $ \lbrace e_1, \ldots, e_n \rbrace $.
Let $ \lambda_i = e_1 + \cdots + e_i - \rho $ where $ \rho $ is half the sum of the positive roots.
Let $ \mathcal{O}_i (\mathfrak{gl}_n) $ be the full subcategory of $ \mathcal{O}_i (\mathfrak{gl}_n) $ consisting of modules which have generalized central character $\chi_{\la_i}$ corresponding to $\lambda_i$ under the Harish-Chandra isomorphism.
The category  $\mathcal{O}_i (\mathfrak{gl}_n)$ is abelian, $\mathbb{C}$-linear with finite dimensional hom-spaces and has enough projective objects. For each $i$, fixing a minimal projective generator $P=P_i$ provides an equivalence of categories between $\mathcal{O}_i (\mathfrak{gl}_n)$     and the category of finitely generated right modules over the complex finite dimensional algebra $A:=\End_\cO(P)$. This algebra can be equipped with a natural Koszul grading. The category of finite dimensional graded $A$-modules can therefore be viewed as a graded version of $\mathcal{O}_i (\mathfrak{gl}_n)$ and is denoted by $ {}^{\mathbb{Z}} \mathcal{O}_i (\mathfrak{gl}_n) $.
The Grothendieck groups $ K({}^{\mathbb{Z}} \mathcal{O}_i (\mathfrak{gl}_n)) $, for $ i=1,\ldots,n$ are naturally $ \mathbb{Z}[q,q^{-1}]$-modules, where $q$ and $q^{-1}$ act by shifting up resp. down the degree by one.

Let  $\mathcal{U}_q(\mathfrak{sl}_2)$ be the generic quantum group of $\mathfrak{sl}_2$, that is the $\mathbb{C}(q)$-algebra with generators $E,F,K^{\mp}$ and the usual relations. Let $V_i$ denote the irreducible representation (of type I) of dimension $i+1$, see e.g. \cite{FSS1} for details.
 As mentioned already in the introduction, the $n$-th tensor power $V^{\otimes n}= V_1^{\otimes n}$ of the $2$-dimensional irreducible type I representation $V_1$ of $ \mathcal{U}_q(\mathfrak{sl}_2) $  (the quantum version of the vector representation) arises as the Grothendieck group $K_0$ of a direct sum of graded versions of blocks of category $\cO$:

\begin{prop}
\label{catoftensors}
\cite[Special case of Theorem 4.1]{FKS}
There exists a canonical isomorphism of  $\mathcal{U}_q(\mathfrak{sl}_2)$-modules
$$ \bigoplus_{i=0}^n \mathbb{C}(q) \otimes_{\mathbb{Z}[q,q^{-1}]} K_0({}^{\mathbb{Z}} \mathcal{O}_i (\mathfrak{gl}_n)) \cong V_1^{\otimes n}. $$
where the action on the left hand side is induced from certain exact functors.
\end{prop}

There is also a categorification of $ V_n$, the $(n+1)$-dimensional, irreducible $ \mathcal{U}_q(\mathfrak{sl}_2) $-module of type I.  Let $ \op{Gr}(i,n) $ be the variety of complex $ i $-dimensional subspaces in $ \mathbb{C}^n$ and denote its singular cohomology with complex coefficients by $C^i = H^*(\op{Gr}(i,n)) $.  Denote the category of finitely generated, graded, right $ C^i $-modules by $ \gmod-C^i$.

\begin{prop}\cite[Theorem 6.2]{FKS}
\label{catofirr}
There is a canonical isomorphism of $\mathcal{U}_q(\mathfrak{sl}_2)$-modules
$$ \bigoplus_{i=0}^n \mathbb{C}(q) \otimes_{\mathbb{Z}[q,q^{-1}]} K_0(\gmod-C^i) \cong V_n .$$
where the action on the left hand side is induced from exact restriction and induction functors.
\end{prop}

The category appearing in Proposition ~\ref{catofirr} is obtained from the category in Proposition ~\ref{catoftensors} by taking a quotient by a Serre subcategory. We have the corresponding quotient functor
\begin{equation}
\label{pin}
\widetilde{\pi}_n := \bigoplus_i^n \widetilde{\pi}_{i,n} \colon \quad\bigoplus_i^n {}^{\mathbb{Z}} \mathcal{O}_i (\mathfrak{gl}_n) \rightarrow \bigoplus_i^n \gmod-C^i
\end{equation}
and its (right exact) left adjoint
\begin{equation}
\label{iotan}
\widetilde{\iota}_n := \bigoplus_i^n \widetilde{\iota}_{i,n} \colon\quad\bigoplus_i^n \gmod-C^i \rightarrow \bigoplus_i^n {}^{\mathbb{Z}} \mathcal{O}_i (\mathfrak{gl}_n)
\end{equation}
The first functor is just a graded version of Soergel's combinatorial functor $\mathbb{V}$ introduced originally in \cite{Soperv}. Since the left adjoint $ \widetilde{\iota}_{n} $ is not exact, we will need its derived functor acting on a certain derived category. Note that we could directly, see \cite[Theorem 3.2]{MiSerre}, work with the one-sided unbounded derived category and consider the derived Verdier quotient. However, basically due to an Eilenberg-Schwindel argument, the Grothendieck groups behave not as we want them to, \cite{Miy}. Hence we will work with slightly smaller categories which whose definition is taken from ~\cite[Section 2.12]{BGS} (alternatively we could also work with the category $D^{\triangledown}(?)$ studied in \cite{AcharStroppel}).

Let $R = \oplus_{j \geq 0} R_j $ be a positively graded ring of finite dimension and $\gmod-R$ the category of finite dimensional graded right $R$-modules.  Let $ \mathcal{K}(\gmod-R) $ be the homotopy category of complexes in $\gmod-R$.  An object $ M \in \mathcal{K}(\gmod-R) $ is
a complex of graded $ R$-modules where the module sitting in the $i-th$ homological degree is $ M^i=\oplus_j M^i_j$.
Let $ \mathcal{K}^<(\gmod-R) $ be the full subcategory of $ \mathcal{K}(R) $ whose objects $ M $ satisfy $ M^i_j =0$ if $ i\gg 0 $ or  $i+j \ll 0$.
Let $ \mathcal{K}^>(\gmod-R) $ be the full subcategory of $ \mathcal{K}(R) $ whose objects $ M $ satisfy $ M^i_j =0$ if $ i\gg 0 $ or  $i+j \gg 0$.
Let $ \mathcal{D}^<(\gmod-R) $ and $ \mathcal{D}^>(\gmod-R) $ be the localizations of $ \mathcal{K}^<(\gmod-R) $ and $ \mathcal{K}^>(\gmod-R) $ respectively at quasi-isomorphisms.

We have the derived functors of the functors appearing in \eqref{pin} and \eqref{iotan}:
\begin{equation}
\label{derivedpin}
\mathcal{R}\widetilde{\pi}_n = \widetilde{\pi}_n := \bigoplus_i^n (\mathcal{R}\widetilde{\pi}_{i,n}=\widetilde{\pi}_{i,n}) \colon\quad \bigoplus_i^n \mathcal{D}^<({}^{\mathbb{Z}} \mathcal{O}_i (\mathfrak{gl}_n)) \rightarrow \bigoplus_i^n \mathcal{D}^<(\gmod-C^i)
\end{equation}
\begin{equation}
\label{derivediotan}
\mathcal{L} \widetilde{\iota}_n = \bigoplus_i^n \mathcal{L} \widetilde{\iota}_{i,n} \colon\quad \bigoplus_i^n \mathcal{D}^{<}(\gmod-C^i) \rightarrow \bigoplus_i^n \mathcal{D}^<({}^{\mathbb{Z}} \mathcal{O}_i (\mathfrak{gl}_n))
\end{equation}
which induce projection and inclusion on the Grothendieck groups, by \cite{FSS1}.

To get the categorified Jones-Wenzl projector, we define
\begin{equation}
\mathbb{P}_n = \oplus_{i=0}^n \mathbb{P}_{i,n} = \oplus_{i=0}^n (\mathcal{L} \widetilde{\iota}_{i,n}) \circ \widetilde{\pi}_{i,n}  \colon\quad \oplus_{i=0}^n D^{<}({}^{\mathbb{Z}} \mathcal{O}_i (\mathfrak{gl}_n)) \rightarrow
\oplus_{i=0}^n D^{<}({}^{\mathbb{Z}} \mathcal{O}_i (\mathfrak{gl}_n))
\end{equation}
the composition of the derived functors in \eqref{derivedpin} and \eqref{derivediotan}.
It was shown in ~\cite{FSS1} that this functor categorifies the Jones-Wenzl projector $ p_n $ and naturally commutes with functors categorifying the action of $ \mathcal{U}_q(\mathfrak{sl}_2) $.
Note that $ \mathcal{L}(\widetilde{\iota}_{i,n}) \circ \widetilde{\pi}_{i,n} \neq \mathcal{L}(\widetilde{\iota}_{i,n} \circ \widetilde{\pi}_{i,n}) $ which makes $ \mathbb{P}_n $ an interesting highly non-trivial functor.

The Koszul duality functor:
\begin{equation*}
\mathbb{D}_{i,n} \colon\quad D^{<}({}^{\mathbb{Z}} \mathcal{O}_i (\mathfrak{gl}_n)) \rightarrow D^{>}({}^{\mathbb{Z}} \mathcal{O}^i (\mathfrak{gl}_n))
\end{equation*}
a triangulated functor from the derived category of graded singular category $ \mathcal{O} $ into the derived category of graded parabolic category $ \mathcal{O}$ was constructed in \cite[Theorem 2.12.1]{BGS} (see also ~\cite{MOS}).
This is a covariant functor whose definition in the special case of $ n=2$ and $ i=1 $ we recall in Section ~\ref{D12}.
It maps simple and injective modules to projective and simple modules respectively ~\cite[Theorem 2.12.5]{BGS}.
A shift in the internal grading is mapped under $ \mathbb{D}_{i,n} $ to a diagonal shift in the internal and homological grading.
More precisely:
\begin{equation}
\label{internalshift}
\mathbb{D}_{i,n} (M\langle r \rangle) \cong (\mathbb{D}_{i,n} M) \langle -r \rangle [-r]
\end{equation}
\begin{equation}
\label{homshift}
\mathbb{D}_{i,n} (M[r]) \cong (\mathbb{D}_{i,n}M)[r]
\end{equation}
where $ (M \langle r \rangle)_j = M_{j-r} $ and $ (M[r])^j = M^{j+r}$ ~\cite[Theorem 2.12.5]{BGS}.

\section{Categorifications of the second Jones-Wenzl projector}
Consider now the case $ \mathfrak{gl}_2 $ from the general setup in Section ~\ref{lietheory}.  Since we only study the case when $ n=2 $ and $ i=1 $, we will omit these indices in what follows.
\subsection{The functor $ \mathbb{P}$}
\label{p12}

Let $ Q $ denote the quiver in ~\eqref{quiversl2} with vertices $1$ and $2$ and arrows $a$, $b$ as indicated.
A path (of length $ l>0$) is a sequence $ p = \alpha_1 \alpha_2 \cdots \alpha_l $ of arrows where the starting point of $ \alpha_i $ is the ending point of $ \alpha_{i+1} $ for $ i=1,\ldots,l-1$.
By $ \mathbb{C}Q $ we denote the path algebra of $ Q $, with basis the set of all paths with additionally $e(1)$ and $e(2)$ the trivial paths of length $ 0 $ beginning at $ 1 $ and $ 2 $ respectively, and product given by concatenation.  For example, $c = ab$ is a basis element of $\mathbb{C}Q$.
The path algebra is a graded algebra where the grading comes from the length of each path.
\begin{equation}
\label{quiversl2}
\begin{tikzpicture}
\filldraw[black](0,0) circle (1pt);
\filldraw[black](1,0) circle (1pt);
\draw (0,0) .. controls (.5,.5) .. (1,0)[->][thick];
\draw (0,0) .. controls (.5,-.5) .. (1,0)[<-][thick];
\draw (.5,.6) node{$a$};
\draw (.5,-.6) node{$b$};
\draw (0,-.3) node{$1$};
\draw (1,-.3) node{$2$};
\end{tikzpicture}
\end{equation}

Set $ B $ to be the algebra $ \mathbb{C}Q $ modulo the two-sided ideal generated by $ ba $.  By abuse of notation, we denote the image of an element $ p \in \mathbb{C}Q $ in the algebra $ B $ also by $ p$.
The algebra $ B $ inherits a grading from $ \mathbb{C}Q$ since the relation $ ba=0 $ is homogenous.  Let $ B_j $ denote degree $ j $ subspace of $ B $.  The degree zero part $ B_0 $ is a semi-simple algebra
spanned by $ e(1) $ and $ e(2)$.  The degree one subspace is spanned by $ a $ and $ b $.  The degree two subspace is spanned by $ c $ and $ B_j=0 $ for all $ j \geq 3$.  Let $ B_+ $ be the subspace of $ B $ whose homogenous elements are in positive degree.  The subspace $ B_+ $ is the radical of $ B $.

The graded category $ {}^{\mathbb{Z}} \mathcal{O}_{1}(\mathfrak{gl}_2) $ is equivalent to the category of finitely generated, graded, right modules over the path algebra of $ B $.
The projective modules $ P(1) = e(1)B $ and $ P(2)=e(2)B $ correspond to the dominant and anti-dominant projective modules respectively in category $ \mathcal{O}_1(\mathfrak{gl}_2)$.  The simple quotients of the latter two objects correspond to one-dimensional right modules over $ B$: $ L(1) = e(1) B/e(1)B_+ $ and $L(2) = e(2)B/e(2)B_+$.

Up to scalars, the only non-zero homogeneous maps between indecomposable projectives objects are:
\begin{eqnarray*}
&b \colon\quad P(2) \langle i \rangle \rightarrow P(1) \langle i-1 \rangle,\quad
a \colon\quad P(1) \langle i \rangle \rightarrow P(2) \langle i-1 \rangle,&\\
&c \colon\quad P(2) \langle i \rangle \rightarrow P(2) \langle i-2 \rangle,&\\
&e(1) \colon\quad P(1) \langle i \rangle \rightarrow P(1) \langle i \rangle,\quad
e(2) \colon\quad P(2) \langle i \rangle \rightarrow P(2) \langle i \rangle&
\end{eqnarray*}
where the maps are multiplication on the left by $ b$, $ a$ , $c$, $e(1)$, and $ e(2) $ respectively.

Let $ C = \End_{\gmod-B_{1,2}}(P(2)) \cong \mathbb{C}[x]/(x^2)$ where $ e(2) \mapsto 1 $ and $ c \mapsto x $.  This becomes an isomorphism of graded algebras if we put $ x $ in degree two and we may identify
$ C $ with $ H^*(\mathbb{P}^1{\mathbb{C}}) $.

Then we have an exact quotient functor $ \widetilde{\pi} \colon\quad \gmod-B \rightarrow \gmod-C $ given by:
$$ M \mapsto \Hom_{\gmod-B}(P(2),M) \langle -1 \rangle. $$
It is clear that $ \widetilde{\pi}(P(2)) = C \langle -1 \rangle $.  The map sending $ e(2) $ to $ a $ and every other basis element to zero gives that $ \widetilde{\pi} (P(1)) \cong \mathbb{C}=C/C_+ \cong \mathbb{C}$.
\label{kd}
There is also a right exact inclusion functor $ \widetilde{\iota} \colon\quad \gmod-C \rightarrow \gmod-B $ given by:
$$ M \mapsto M \otimes_{C} P(2) \langle 1 \rangle. $$
Clearly $ \widetilde{\iota} (C) \cong P(2) \langle 1 \rangle$.

Since the functor $ \widetilde{\iota}$ is not exact, we take derived functors and consider the composite
\begin{equation*}
\mathbb{P} =  \mathcal{L} \widetilde{\iota} \circ \widetilde{\pi} \colon\quad D^{<}(\gmod-B) \rightarrow D^{<}(\gmod-B).
 \end{equation*}

It follows that:
\begin{equation}
\label{ponP(2)}
\mathbb{P} (P(2)) = P(2)
\end{equation}
where the equality means canonical isomorphism, and
\begin{equation}
\label{ponP(1)}
\mathbb{P}(P(1)) \cong \cdots \rightarrow P(2) \langle 5 \rangle \rightarrow P(2) \langle 3 \rangle \rightarrow P(2) \langle 1 \rangle \rightarrow 0
\end{equation}
where the object $ P(2) \langle 1 \rangle $ is in homological degree zero and all the maps are $ c $.

\subsection{The Koszul duality functor $ \mathbb{D}$.}
\label{D12}
First we define the Koszul dual algebra $ B^! $ of the algebra $ B $ and then write down the functor $ \mathbb{D} $ following ~\cite{BGS}.
Let $ V = B_1 $ which is spanned by $ a $ and $ b $.  The algebra $ B $ is a quadratic algebra defined as the tensor algebra of $ V $ over $ B_0 $ modulo the
ideal generated by $ b \otimes a $.
Set $ V^* = \Hom_{B_0-\Mod}(V, B_0)$ which is spanned by $ a^*, b^*$ where
\begin{equation*}
a^*(a) = e(2),\quad
a^*(b) = 0,\quad
b^*(a) = 0,\quad
b^*(b) = e(1)
\end{equation*}
This space is a $ (B_0, B_0)$-bimodule where
\begin{equation*}
e(2)a^*=0,
\hspace{.2in}
e(1)a^*=a^*,
\hspace{.2in}
e(2)b^*=b^*,
\hspace{.2in}
e(1)b^*=0,
\end{equation*}
\begin{equation*}
a^*e(1)=0,
\hspace{.2in}
a^*e(2)=a^*,
\hspace{.2in}
b^*e(1)=b^*,
\hspace{.2in}
b^*e(2)=0.
\end{equation*}
Then $ B^! $ is defined to be the tensor algebra of $ V^* $ over $ B_0 $ modulo the ideal generated by $ a^* \otimes b^* $.

Consider the quiver $ Q^! $:
\begin{equation}
\label{quiversl2dual}
\begin{tikzpicture}
\filldraw[black](0,0) circle (1pt);
\filldraw[black](1,0) circle (1pt);
\draw (0,0) .. controls (.5,.5) .. (1,0)[->][thick];
\draw (0,0) .. controls (.5,-.5) .. (1,0)[<-][thick];
\draw (.5,.6) node{$b^*$};
\draw (.5,-.6) node{$a^*$};
\draw (0,-.3) node{$1$};
\draw (1,-.3) node{$2$};
\end{tikzpicture}
\end{equation}
It is clear that $ B^! $ is isomorphic to the quotient of the path algebra $ \mathbb{C}Q^! $ modulo the ideal is generated by $ a^* b^* $.
There is an obvious isomorphism $ \varphi\quad B \rightarrow B^! $ where
\begin{equation*}
e(1) \mapsto e(2),
\hspace{.2in}
e(2) \mapsto e(1),
\hspace{.2in}
a \mapsto a^*,
\hspace{.2in}
b \mapsto b^*,
\hspace{.2in}
ab \mapsto a^*b^*
\end{equation*}

An object $ M $ in $ \mathcal{D}^{<}(\gmod-B) $ can be viewed as a bigraded vector space $ \oplus_{r,s} M^r_s $ with a differential $ d \colon M^r_s \rightarrow M^{r+1}_s $ for all $ r $ and $ s $. Using the isomorphism $ \varphi $, following \cite{BGS}, we set
\begin{equation}
(\mathbb{D} M)^p_q = \bigoplus_{\substack{p=r+s\\q=l-s}} M^r_s \otimes_{B_0} B_l
\end{equation}
where the left action of $ B_0 $ is twisted by the isomorphism $ \varphi $.

Define
$ \mathbb{D} M = \bigoplus_{p,q} (\mathbb{D} M)^p_q $
to be a graded complex with differential $ \tilde{d} $ given by
\begin{equation}
\tilde{d}(m \otimes g) = dm \otimes g + (-1)^{r+s}(ma \otimes ag + mb \otimes bg)
\end{equation}
where $ m \in M^r_s $. Then we get a functor
\begin{equation*}
\mathbb{D} \colon\quad \mathcal{D}^{<} (\gmod-B) \rightarrow \mathcal{D}^{>} (\gmod-B).
\end{equation*}

Let $ I(2) $ be the injective hull of $ L(2)$.
We will need the following easy facts which can be found in \cite{BGS}.

\begin{lemma}There are isomorphisms of graded modules
\begin{equation}
\label{kdm}
\mathbb{D} L(1) \cong P(2),\quad
\mathbb{D} L(2) \cong P(1),\quad
\mathbb{D} I(2) \cong L(1),\quad
P(2) \cong I(2) \langle 2 \rangle,\quad
\end{equation}
\end{lemma}

{\bf Example: $ \mathbb{D} L(1) \cong P(2) $.}
Since $ L(1)^r_s = 0 $ unless $ r=s=0$, if $ p \neq 0 $, then $ (\mathbb{D} L(1))^p_q = 0 $.
By definition, $ (\mathbb{D} L(1))^0_q = L(1) \otimes_{B_0} B_q $.
Clearly, $ L(1) \otimes_{B_0} B_0 $ is spanned by $ e(1) \otimes e(2)$,
$ L(1) \otimes_{B_0} B_1 $ is spanned by $ e(1) \otimes a$, and
$ L(1) \otimes_{B_0} B_2 $ is spanned by $ e(1) \otimes ab$.
This is clearly isomorphic to $ P(2)$.

\subsection{The functor $ \mathbb{D} \circ \mathbb{P}$}

First we take the Koszul dual of the complexes obtained in Section ~\ref{p12}.
Using \eqref{internalshift}, \eqref{ponP(2)} and \eqref{kdm} we get:
\begin{equation}
\mathbb{D} \circ \mathbb{P}(P(2)) = \mathbb{D}(P(2)) \cong L(1)\langle -2 \rangle [-2].
\end{equation}
Note that $ \mathbb{P}(P(1)) $ is an object of $ \mathcal{D}^{<}(\gmod-B) $ by \eqref{ponP(1)}.  Then we get:
\begin{equation}
\label{KDBGP(10)}
\mathbb{D}_{1,2} \circ \widetilde{p}_{1,2}(P(1)) \cong ( \cdots \rightarrow L(1) \langle -4 \rangle [-4] \rightarrow L(1) \langle -2 \rangle [-2] \rightarrow L(1) \rightarrow 0) \langle -3 \rangle [-3]
\end{equation}
where the rightmost $ L(1) $ sits in homological degree zero.

It is easy to see that the simple object $L(1)$ is quasi-isomorphic to its minimal projective resolution:
\begin{equation}
\xymatrix{
0 \ar[r] & P(1) \langle 2\rangle \ar[r]^{a} & P(2)\langle 1\rangle \ar[r]^{b} & P(1) \rightarrow 0.
}
\end{equation}

Then \eqref{KDBGP(10)} is the total complex shifted by $ \langle -3 \rangle [-3] $ of the following complex of complexes where the object $ P(1) \langle 2 \rangle $ in the upper-right corner sitting in homological degree $-2$:

\begin{equation}
\label{complexofcomplexes}
\xymatrix{
& & & 0 \ar[d] \\
& & & P(1) \langle 2 \rangle \ar[d]\\
& & 0\ar[d] & P(2) \langle 1 \rangle \ar[d] \\
&& P(1) \ar[d] \ar[r] & P(1) \ar[d] \\
& 0 \ar[d] & P(2) \langle -1 \rangle \ar[d] &0 \\
& P(1) \langle -2 \rangle \ar[d] \ar[r] & P(1) \langle -2 \rangle \ar[d] & \\
& P(2) \langle -3 \rangle\ar[d] & 0 &  \\
\cdots \ar[r] & P(1) \langle -4 \rangle \ar[d] & & \\
& 0 & &
}
\end{equation}

The total complex of \eqref{complexofcomplexes} is the middle complex in \eqref{sl2comp} where the object $ P(1) \langle 2 \rangle $ in the bottom row sits in homological degree $-2$.   The complex is an object of
$ D^{>}(\gmod-B)$.
It is easy to check that this middle complex decomposes as a direct sum of the complexes displayed in the left and right columns in \eqref{sl2comp}.

\begin{figure}
\begin{equation}
\label{sl2comp}
\xymatrix{
\cdots & \cdots  & \cdots \\
P(1) \langle -4 \rangle \oplus P(1) \langle -6 \rangle \ar[u]^{A_4}
\ar[r]^>{\!\!\!\!\!\!\!\!\!\!\!\!J_4}   & P(1) \langle -4 \rangle \oplus P(2) \langle -5 \rangle \oplus P(1) \langle -6 \rangle \ar[u]^{B_4} \ar[r]^>{\!\!\!\!\!\!\!\!\!\!\!\!K_4} \ar@<1ex>[l]^>{ L_4 \!\!\!\!\!\!\!\!\!\!\!\!\!\!\!\!\!\!\!\!\!\!\!\!}& P(2) \langle -5 \rangle\ar[u]^{C_4}\ar@<1ex>[l]^>{M_4 \!\!\!\!\!\!\!\!\!\!\!\!\!\!\!\!\!\!\!\!\!\!\!\!  } \\
P(1) \langle -2 \rangle \oplus P(1) \langle -4 \rangle \ar[u]^{A_3}  \ar[r]^{\!\!\!\!\!\!\!\!\!\!\!\!J_3} & P(1) \langle -2 \rangle \oplus P(2) \langle -3 \rangle \oplus P(1) \langle -4 \rangle \ar[u]^{B_3} \ar[r]^>{\!\!\!\!\!\!\!\!\!\!\!\!K_3} \ar@<1ex>[l]^>{L_3 \!\!\!\!\!\!\!\!\!\!\!\!\!\!\!\!\!\!\!\!\!\!\!\!}& P(2) \langle -3 \rangle\ar[u]^{C_3}\ar@<1ex>[l]^>{M_3 \!\!\!\!\!\!\!\!\!\!\!\!\!\!\!\!\!\!\!\!\!\!\!\! } \\
P(1) \oplus P(1) \langle -2 \rangle \ar[u]^{A_2}  \ar[r]^>{\!\!\!\!\!\!\!\!\!\!\!\!\!\!\!\!\!\!\!\!\!\!\!\!\!J_2} & P(1) \oplus P(2) \langle -1 \rangle \oplus P(1) \langle -2 \rangle \ar[u]^{B_2} \ar[r]^>{\!\!\!\!\!\!\!\!\!\!\!\!K_2} \ar@<1ex>[l]^>{L_2 \!\!\!\!\!\!\!\!\!\!\!\!\!\!\!\!\!\!\!\!\!\!\!\!\!\!\!\!\!\!\!\!\!\!\!\!\!\!\!\!\!\!\!\!\!\!\!\!}& P(2) \langle -1 \rangle\ar[u]^{C_2} \ar@<1ex>[l]^>{M_2 \!\!\!\!\!\!\!\!\!\!\!\!\!\!\!\!\!\!\!\!\!\!\!\!\!\!\!\!\!\!\!\!\!\!\!\!\!\!\!\!\!\!\!\!\! }\\
P(1) \ar[u]^{A_1} \ar[r]^>{\!\!\!\!\!\!\!\!\!\!\!\!\!\!\!\!\!\!\!\!\!\!\!\!\!\!\!\!\!\!\!\!\!\!\!\!\!\!\!\!\!\!\!\!\!\!\!\!\!\!\!\!\!\!\!\!\!\!\!\!\!\!\!\!\!\!\!\!\!\!\!\!\!\!\!J_1} & P(2) \langle 1 \rangle \oplus P(1) \ar[u]^{B_1}  \ar[r]^>{\!\!\!\!\!\!\!\!\!\!\!\!\!\!\!\!\!\!K_1} \ar@<1ex>[l]^>{L_1 \!\!\!\!\!\!\!\!\!\!\!\!\!\!\!\!\!\!\!\!\!\!\!\!\!\!\!\!\!\!\!\!\!\!\!\!\!\!\!\!\!\!\!\!\!\!\!\!\!\!\!\!\!\!\!\!\!\!\!\!\!\!\!\!\!\!\!\!\!\!\!\!\!\!\!\!\!\!\!\!\!\!\!\!\!\!\!\!\!\!\!\!\!\!\!\!} & P(2) \langle 1 \rangle \ar[u]^{C_1} \ar@<1ex>[l]^>{M_1 \!\!\!\!\!\!\!\!\!\!\!\!\!\!\!\!\!\!\!\!\!\!\!\!\!\!\!\!\!\!\!\!\!\!\!\!\!\!\!\!\!\!\!\!\!\!\!\!\!\!\!\!\!\!\!\!\!\!\!\!\!\!\!\!\!\!\!\!\!\!\!\!\!\!\!\!\!\!\!\!\!\!\!\!\!\!\!\!\!\!\!\!\!\!}\\
0 \ar[u]^{A_0} \ar[r]^>{\!\!\!\!\!\!\!\!\!\!\!\!\!\!\!\!\!\!\!\!\!\!\!\!\!\!\!\!\!\!\!\!\!\!\!\!\!\!\!\!\!\!\!\!\!\!\!\!\!\!\!\!\!\!\!\!\!\!\!\!\!\!\!\!\!\!\!\!\!\!\!\!\!\!\!\!\!\!\!\!\!\!\!\!\!\!\!\!\!\!\!\!\!\!\!\!\!\!\!\!\!\!\!\!\!\!\!\!\!J_0} & P(1) \langle 2 \rangle \ar[u]^{B_0} \ar[r]^>{\!\!\!\!\!\!\!\!\!\!\!\!\!\!\!\!K_0} \ar@<1ex>[l]^>{L_0 \!\!\!\!\!\!\!\!\!\!\!\!\!\!\!\!\!\!\!\!\!\!\!\!\!\!\!\!\!\!\!\!\!\!\!\!\!\!\!\!\!\!\!\!\!\!\!\!\!\!\!\!\!\!\!\!\!\!\!\!\!\!\!\!\!\!\!\!\!\!\!\!\!\!\!\!\!\!\!\!\!\!\!\!\!\!\!\!\!\!\!\!\!\!\!\!\!\!\!\!\!\!\!\!\!} & P(1) \langle 2 \rangle \ar[u]^{C_0} \ar@<1ex>[l]^>{M_0 \!\!\!\!\!\!\!\!\!\!\!\!\!\!\!\!\!\!\!\!\!\!\!\!\!\!\!\!\!\!\!\!\!\!\!\!\!\!\!\!\!\!\!\!\!\!\!\!\!\!\!\!\!\!\!\!\!\!\!\!\!\!\!\!\!\!\!\!\!\!\!\!\!\!\!\!\!\!\!\!\!\!\!\!\!\!\!\!\!\!\!\!\!\!\!\!\!\!\!\!\!\!\!\!\!\!\!\!\!\!\!\!\!\!\!\!\!\!\!\!\!\!\!\!\!\!} \\
& 0 \ar[u] & 0 \ar[u]
}
\end{equation}
\end{figure}

where for $ n \geq 2$,
\begin{equation*}
A_0 = \begin{pmatrix}
0
\end{pmatrix}
\hspace{.2in}
A_{1} = \begin{pmatrix}
1 \\
0
\end{pmatrix}
\hspace{.2in}
A_{n}= \begin{pmatrix}
0 & 1\\
0 & 0
\end{pmatrix}
\end{equation*}

\begin{equation*}
B_0 = \begin{pmatrix}
a \\
0
\end{pmatrix}
\hspace{.2in}
B_{1} = \begin{pmatrix}
b & 1 \\
0 & -a \\
0 & 0
\end{pmatrix}
\hspace{.2in}
B_{n}= \begin{pmatrix}
0 & b & 1\\
0 & 0 & -a\\
0 & 0 & 0
\end{pmatrix}
\end{equation*}

\begin{equation*}
C_0 = \begin{pmatrix}
a
\end{pmatrix}
\hspace{.2in}
C_1 = \begin{pmatrix}
c
\end{pmatrix}
\hspace{.2in}
C_n = \begin{pmatrix}
c
\end{pmatrix}
\end{equation*}

\begin{equation*}
J_0 = \begin{pmatrix}
0
\end{pmatrix}
\hspace{.2in}
J_1 = \begin{pmatrix}
0 \\
1
\end{pmatrix}
\hspace{.2in}
J_n = \begin{pmatrix}
1 & 0\\
-a & 0\\
0 & 1
\end{pmatrix}
\end{equation*}

\begin{equation*}
K_0 = \begin{pmatrix}
1
\end{pmatrix}
\hspace{.2in}
K_1 = \begin{pmatrix}
1 & 0\\
\end{pmatrix}
\hspace{.2in}
K_n = \begin{pmatrix}
a & 1 & 0
\end{pmatrix}
\end{equation*}

\begin{equation*}
L_0 = \begin{pmatrix}
0
\end{pmatrix}
\hspace{.2in}
L_1 = \begin{pmatrix}
b & 1
\end{pmatrix}
\hspace{.2in}
L_n = \begin{pmatrix}
1 & 0 & 0\\
0 & b & 1
\end{pmatrix}
\end{equation*}

\begin{equation*}
M_0 = \begin{pmatrix}
1
\end{pmatrix}
\hspace{.2in}
M_1 = \begin{pmatrix}
1 \\
-b
\end{pmatrix}
\hspace{.2in}
M_n = \begin{pmatrix}
0\\
1\\
-b
\end{pmatrix}
\end{equation*}

Hence we have determined $ \mathbb{D} \circ \mathbb{P} $ on projective objects in $ \gmod-B $ and will now determine it on maps between projective objects.



\begin{lemma}
\begin{fleqn}
\label{Dpc}
\begin{align*}
\mathbb{D}_{1,2} \circ \widetilde{p}_{1,2} (c) \colon\quad
\end{align*}
\end{fleqn}
\begin{equation*}
\xymatrix{
  & & 0 \ar[r] & P(1) \langle -2 \rangle \ar[r]^{a} \ar[d]^{1}& P(2) \langle -3 \rangle \ar[r]^{b} & P(1) \langle -4 \rangle \ar[r] & 0\\
 0 \ar[r] & P(1) \ar[r]^{a} & P(2) \langle -1 \rangle \ar[r]^{b} & P(1) \langle -2 \rangle \ar[r] & 0
}
\end{equation*}

\begin{fleqn}
\label{Dpa}
\begin{align*}
\mathbb{D}_{1,2} \circ \widetilde{p}_{1,2} (a) \colon\quad
\end{align*}
\end{fleqn}
\begin{equation*}
\xymatrix{
& & 0 \ar[r] & P(1) \langle -2 \rangle \ar[r]^{a} \ar[d]^{1} & P(2) \langle -3 \rangle \ar[r]^{c} & P(2) \langle -5 \rangle \ar[r]^{c} & \cdots \\
0 \ar[r] & P(1) \ar[r]^{a} & P(2) \langle -1 \rangle \ar[r]^{b} & P(1) \langle -2 \rangle \ar[r] & 0
}
\end{equation*}

\begin{fleqn}
\label{Dpb}
\begin{align*}
\mathbb{D}_{1,2} \circ \widetilde{p}_{1,2} (b) \colon\quad
\end{align*}
\end{fleqn}
\begin{equation*}
\xymatrix{
0 \ar[r] & P(1) \langle -1 \rangle \ar[r]^{a} \ar[d]^{1} & P(2) \langle -2 \rangle \ar[r]^{b} \ar[d]^{1} & P(1) \langle -3 \rangle \ar[r] \ar[d]^{a} & 0 & \\
0 \ar[r] & P(1) \langle -1 \rangle \ar[r]^{a} & P(2) \langle -2 \rangle \ar[r]^{c} & P(2) \langle -4 \rangle \ar[r]^{c} & P(2) \langle -6 \rangle \ar[r]^{c} & \cdots
}
\end{equation*}
\end{lemma}

\subsection{The Cooper-Krushkal projector $ \mathbb{CK}$}
Let $ \theta $ be the $ (B, B)- $ bimodule $ Be(2) \otimes e(2) B \langle -1 \rangle$.
By abuse of notation, if $ M $ is an object of $ \gmod-B $, then $ \theta M $ is an object of $ \gmod-B $ given by $ M \otimes_{B} \theta $ giving rise to an exact functor.
\begin{remark}{\rm
This bimodule was studied in \cite{KS} as a categorified generator of the two-dimensional Temperley-Lieb algebra $TL_2 $ and Lie theoretically in \cite{Strgrad} as a translation functor where the natural transformations were also considered.
}
\end{remark}

There are bimodule homomorphisms:
\begin{eqnarray*}
\alpha \colon&& B \langle i \rangle \rightarrow \theta \langle i-1 \rangle \hspace{.2in} e(2) \mapsto c \otimes e(2) + e(2) \otimes c \hspace{.2in} e(1) \mapsto a \otimes b\\
\beta \colon&& \theta \langle i \rangle \rightarrow \theta \langle i-2 \rangle \hspace{.2in} e(2) \otimes e(2) \mapsto c \otimes e(2) - e(2) \otimes c\\
\gamma \colon&& \theta \langle i \rangle \rightarrow \theta \langle i-2 \rangle \hspace{.2in} e(2) \otimes e(2) \mapsto c \otimes e(2) + e(2) \otimes c
\end{eqnarray*}

\begin{prop}
\label{cksl2}
There is a complex of bimodules as follows $ \mathbb{CK}:= $
\begin{displaymath}
\xymatrix{0 \ar[r] & B_{1,2} \ar[r]^{\alpha} & \theta_1 \langle -1 \rangle \ar[r]^{\beta} & \theta_1 \langle -3 \rangle \ar[r]^{\gamma} & \theta_1 \langle -5 \rangle \ar[r]^{\beta} & \theta_1 \langle -7 \rangle \ar[r]^{\gamma} & \cdots}
\end{displaymath}
\end{prop}

\begin{proof}
It is routine to check that the composition of consecutive maps is zero.
\end{proof}

Let $ X $ be a complex in $ \mathcal{D}^>(\gmod-B) $.  Then $ X \otimes_B \mathbb{CK} $ is a bicomplex of $ B $-modules.  This gives rise to a functor
\begin{equation*}
\mathbb{CK} \colon\quad \mathcal{D}^>(\gmod-B) \rightarrow \mathcal{D}^>(\gmod-B)
\end{equation*}
where the complex $ X $ is mapped to be the total complex of $ X \otimes_B \mathbb{CK} $.

The functor $\mathbb{CK}$ is a Cooper-Krushkal universal projector of width two up to a renormalization in the grading: our internal grading is opposite to that of \cite{CoKr}.

\begin{lemma} $\Theta$ acts on projectives as follows
\label{p1onprojectives}
\begin{enumerate}
\item $ \theta P(2) \cong P(2) \langle -1 \rangle \oplus P(2) \langle 1 \rangle $
\item $ \theta P(1) \cong P(2)$.
\end{enumerate}
\end{lemma}

\begin{proof}
Under the equality:
\begin{equation*}
\theta P(2) = e(2) B \otimes_{B} B e(2) \otimes e(2) B \langle -1 \rangle,
\end{equation*}
the first isomorphism maps $ 1 \otimes 1 \otimes y $ to $ y $ in the first coordinate and $ c \otimes 1 \otimes y $ to $ y $ in the second coordinate. For the second part of the lemma, first note that
\begin{equation*}
\theta P(1) = e(1) B \otimes_{B} B e(2) \otimes e(2) B \langle -1 \rangle.
\end{equation*}
The desired isomorphism maps $ a \otimes y $ to $ y$.
\end{proof}

\begin{prop}
\label{CKonprojectives}
There are isomorphisms in $ \mathcal{D}^>(\gmod-B)$:
\begin{enumerate}
\item $ \mathbb{CK}(P(2)) \cong 0 $
\item $ \mathbb{CK}(P(1)) \cong $
\begin{displaymath}
\xymatrix{0 \ar[r] & P(1)\ar[r]^{b} & P(2) \langle -1 \rangle \ar[r]^{-c} & P(2) \langle -3 \rangle \ar[r]^{c} & P(2) \langle -5 \rangle \ar[r]^{-c} & \cdots}
\end{displaymath}
\end{enumerate}
\end{prop}

\begin{proof}
Applying the functor $ \mathbb{CK} $ to the module $ P(2) $ gives the complex which is up to shift 2-periodic after the first step:

\begin{displaymath}
\xymatrix{0 \ar[r] &P(2) \ar[r]^{\nu} &M \ar[r]^{\eta} &M \langle -2 \rangle \ar[r]^{\zeta} &M \langle -4 \rangle \ar[r]^{\eta} & \cdots}
\end{displaymath}

where $ M = P(2) \langle -2 \rangle \oplus P(2) $ and

\begin{equation*}
\nu = \begin{pmatrix}
c \\
1
\end{pmatrix}
\hspace{.2in}
\eta = \begin{pmatrix}
-c & 0\\
1 & -c
\end{pmatrix}
\hspace{.2in}
\zeta = \begin{pmatrix}
c & 0\\
1 & c
\end{pmatrix}
\end{equation*}
This is clearly homotopically trivial.
The second isomorphism follows from Lemma ~\ref{p1onprojectives}.
\end{proof}



\subsection{A duality theorem}
\label{dualitytheoremsection}
\begin{lemma}
\label{dualofP(1)}
$ \mathbb{D}P(1) $ is quasi-isomorphic to:
\begin{equation}
\xymatrix{
0 \ar[r] & \ar[r] P(2) \ar[r]^{b} & P(1) \langle -1 \rangle \ar[r] & 0
}
\end{equation}
where the object $ P(1) \langle -1 \rangle $ is in homological degree one.
\end{lemma}

\begin{proof}
There is an exact sequence of modules:
$$ 0 \rightarrow L(2) \langle 1 \rangle \rightarrow P(1) \rightarrow L(1) \rightarrow 0 $$
which gives rise to a triangle in the derived category:
$$ L(1)[-1] \rightarrow L(2) \langle 1 \rangle \rightarrow P(1). $$
Applying the functor $ \mathbb{D} $ to this triangle gives the triangle:
$$ P(2)[-1] \rightarrow P(1)[-1] \langle -1 \rangle \rightarrow \mathbb{D}P(1). $$
\end{proof}

\begin{lemma}
\label{a}
\begin{enumerate}
\item $ \mathbb{CK} \circ \mathbb{D} (P(2)) \cong L(1) \langle -2 \rangle [-2] $
\item $ \mathbb{CK} \circ \mathbb{D} (P(1)) \cong $
\begin{displaymath}
\xymatrix{0 \ar[r] & P(1) \langle -1 \rangle \ar[r]^{b} & P(2) \langle -2 \rangle \ar[r]^{-c} & P(2) \langle -4 \rangle \ar[r]^{c} & P(2) \langle -6 \rangle \ar[r]^{-c} & \cdots}
\end{displaymath}
where the object $ P(1) $ is in homological degree $ 1$.
\end{enumerate}
\end{lemma}

\begin{proof}
Recall $ P(2) $ isomorphic to $ I(2) \langle 2 \rangle $ where $ I(2) $ is the injective hull of $ L(2) $ and $ \mathbb{D} I(2) \cong L(1) $.
Then tensoring $ \mathbb{CK} $ with $ 0 \rightarrow P(1) \langle 2 \rangle \rightarrow P(2) \langle 1 \rangle \rightarrow P(1) \rightarrow 0 $ (over $ B$) using Proposition ~\ref{CKonprojectives}, gives the following double complex:
\begin{equation}
\label{aa}
\xymatrix{
& \cdots & \cdots  & \cdots \\
0 \ar[r] & P(2) \langle -3 \rangle \ar[u]^{-c} \ar[r] & 0 \ar[r] \ar[u] & P(2) \langle -5 \rangle \ar[u]^{-c} \ar[r] & 0\\
0 \ar[r] & P(2) \langle -1 \rangle \ar[u]^{c} \ar[r] & 0 \ar[r] \ar[u] & P(2) \langle -3 \rangle \ar[u]^{c} \ar[r] & 0\\
0 \ar[r] & P(2) \langle 1 \rangle \ar[u]^{-c} \ar[r] & 0 \ar[r] \ar[u] & P(2) \langle -1 \rangle \ar[u]^{-c} \ar[r] & 0\\
0 \ar[r] & P(1) \langle 2 \rangle \ar[u]^{b} \ar[r] & 0 \ar[r] \ar[u] & P(1) \ar[u]^{b} \ar[r] & 0\\
& 0 \ar[u] & 0 \ar[u] & 0 \ar[u]
}
\end{equation}

It is easy to see that its double complex is homotopically equivalent to
\begin{equation*}
0 \rightarrow P(1) \langle 2 \rangle \rightarrow P(2) \langle 1 \rangle \rightarrow P(1) \rightarrow 0
\end{equation*}
which is isomorphic to $ L(1) \langle -2 \rangle [-2]$ giving the first isomorphism.

The second isomorphism follows easily from Lemma ~\ref{dualofP(1)} and \eqref{CKonprojectives}.
\end{proof}

Now we summarize the calculations of the functors on projective objects:

\begin{prop}
\label{functorsonprojectives}
\begin{equation*}
\mathbb{D} \circ \mathbb{P}(P(1)) \cong \mathbb{CK} \circ \mathbb{D}(P(1))
\end{equation*}
\begin{equation*}
\mathbb{D} \circ \mathbb{P}(P(2)) \cong \mathbb{CK} \circ \mathbb{D}(P(2))
\end{equation*}
\end{prop}

Now we check that
$ \mathbb{D} \circ \mathbb{P} $ equals $ \mathbb{CK} \circ \mathbb{D} $
on maps.

\begin{lemma}
\label{functoronc}
$ \mathbb{D} \circ \mathbb{P}(c)  = \mathbb{CK} \circ \mathbb{D}(c). $
\end{lemma}

\begin{proof}
The morphism $ c \colon P(2) \langle 2 \rangle \rightarrow P(2) $ gets mapped under Koszul duality to the morphism $ \mathbb{D}(c) \colon\quad $
\begin{equation*}
\xymatrix{
  & & 0 \ar[r] & P(1) \langle -2 \rangle \ar[r] \ar[d]^{1}& P(2) \langle -3 \rangle \ar[r] & P(1) \langle -4 \rangle \ar[r] & 0\\
 0 \ar[r] & P(1) \ar[r] & P(2) \langle -1 \rangle \ar[r] & P(1) \langle -2 \rangle \ar[r] & 0
}
\end{equation*}

Tensoring with the complex $ \mathbb{CK} $ gives a map of bicomplexes as shown in ~\eqref{ccomplex} where the bicomplex $  \mathbb{D}(P(2) \langle 2 \rangle) \otimes_B \mathbb{CK} $ is on the right.
The morphism between complexes is given by the curved arrows and the arrows which would be zero have been omitted.

\begin{equation}
\label{ccomplex}
\xymatrix{
\cdots & \cdots  & \cdots
& \cdots & \cdots  & \cdots \\
P(2) \langle -3 \rangle \ar[u]^{-c} \ar[r] & 0 \ar[r] \ar[u] & P(2) \langle -5 \rangle \ar[u]^{-c}
&P(2) \langle -5 \rangle \ar[u]^{-c} \ar[r] \ar@/^1pc /[l]^{1} & 0 \ar[r] \ar[u] & P(2) \langle -7 \rangle \ar[u]^{-c}  \\
P(2) \langle -1 \rangle \ar[u]^{c} \ar[r] & 0 \ar[r] \ar[u] & P(2) \langle -3 \rangle \ar[u]^{c}
&P(2) \langle -3 \rangle \ar[u]^{c} \ar[r] \ar@/^1pc /[l]^{1} & 0 \ar[r] \ar[u] & P(2) \langle -5 \rangle \ar[u]^{c} \\
P(2) \langle 1 \rangle \ar[u]^{-c} \ar[r] & 0 \ar[r] \ar[u] & P(2) \langle -1 \rangle \ar[u]^{-c}
&P(2) \langle -1 \rangle \ar[u]^{-c} \ar[r] \ar@/^1pc /[l]^{1} & 0 \ar[r] \ar[u] & P(2) \langle -3 \rangle \ar[u]^{-c} \\
P(1) \langle 2 \rangle \ar[u]^{b} \ar[r] & 0 \ar[r] \ar[u] & P(1) \ar[u]^{b}
& P(1) \ar[u]^{b} \ar[r]  \ar@/^1pc /[l]^{1}& 0 \ar[r] \ar[u] & P(1) \langle -2 \rangle \ar[u]^{b} \\
0 \ar[u] & 0 \ar[u] & 0 \ar[u]
&0 \ar[u] & 0 \ar[u] & 0 \ar[u]
}
\end{equation}

This induces
$ \mathbb{CK} \circ \mathbb{D}(P(2) \langle 2 \rangle) \rightarrow \mathbb{CK} \circ \mathbb{D}(P(2)) $ on total complexes:
\begin{equation*}
\xymatrix{
  & & 0 \ar[r] & P(1) \ar[r] \ar[d]^{1}& P(2) \langle -1 \rangle \ar[r] & P(1) \langle -2 \rangle \ar[r] & 0\\
 0 \ar[r] & P(1) \langle 2 \rangle \ar[r] & P(2) \langle 1 \rangle \ar[r] & P(1) \ar[r] & 0
}
\end{equation*}

This is equal to the map $ \mathbb{D} \circ \mathbb{P}(c) $ and the claim follows.
\end{proof}

\begin{lemma}
\label{functorona}
$ \mathbb{D} \circ \mathbb{P}(a)  = \mathbb{CK} \circ \mathbb{D}(a). $
\end{lemma}

\begin{proof}
The morphism $ a \colon P(1) \langle 1 \rangle \rightarrow P(2) $ gets mapped under Koszul duality to the morphism $ \mathbb{D}(a) \colon$
\begin{equation*}
\xymatrix{
& 0 \ar[r] \ar[d] & P(2) \langle -1 \rangle \ar[d]^{1} \ar[r]^{b} & P(1) \langle -2 \rangle \ar[d]^{1} \ar[r] & 0\\
0 \ar[r] & P(1) \ar[r]^{a} & P(2) \langle -1 \rangle \ar[r]^{b} & P(1) \langle -2 \rangle \ar[r] & 0
}
\end{equation*}

Tensoring with the complex $ \mathbb{CK} $ gives rise to the morphism of bicomplexes indicated by the curved arrows in \eqref{acomplex}, where for simplicity all the zero maps are omitted.

\begin{equation}
\label{acomplex}
\xymatrix{
\cdots & \cdots  & \cdots
& \cdots & \cdots \\
P(2) \langle -5 \rangle \ar[u]^{-c} \ar[r] & 0 \ar[r] \ar[u] & P(2) \langle -7 \rangle \ar[u]^{-c}
& 0 \ar[r] \ar[u] &P(2) \langle -7 \rangle \ar[u]^{-c} \ar@/^1pc /[ll]^>{1}   \\
P(2) \langle -3 \rangle \ar[u]^{c} \ar[r] & 0 \ar[r] \ar[u] & P(2) \langle -5 \rangle \ar[u]^{c}
& 0 \ar[r] \ar[u] &P(2) \langle -5 \rangle \ar[u]^{c} \ar@/^1pc /[ll]^>{1}  \\
P(2) \langle -1 \rangle \ar[u]^{-c} \ar[r] & 0 \ar[r] \ar[u] & P(2) \langle -3 \rangle \ar[u]^{-c}
& 0 \ar[r] \ar[u] &P(2) \langle -3 \rangle \ar[u]^{-c}  \ar@/^1pc /[ll]^>{1}  \\
P(1)  \ar[u]^{b} \ar[r] & 0 \ar[r] \ar[u] & P(1) \langle -2 \rangle \ar[u]^{b}
& 0 \ar[r] \ar[u] & P(1) \langle -2 \rangle \ar[u]^{b}  \ar@/^1pc /[ll]^>{1} \\
0 \ar[u] & 0 \ar[u] & 0 \ar[u]
& 0 \ar[r] \ar[u] &0 \ar[u]
}
\end{equation}

It induces the following map
$ \mathbb{CK} \circ \mathbb{D}(P(1) \langle 1 \rangle) \rightarrow \mathbb{CK} \circ \mathbb{D}(P(2)) $ on total complexes:

\begin{equation*}
\xymatrix{
& & 0 \ar[r] & P(1) \langle -2 \rangle \ar[r]^{C_0} \ar[d]^{1} & P(2) \langle -3 \rangle \ar[r]^{C_1} & P(2) \langle -5 \rangle \ar[r]^{C_2} & \cdots \\
0 \ar[r] & P(1) \ar[r]^{} & P(2) \langle -1 \rangle \ar[r] & P(1) \langle -2 \rangle \ar[r] & 0
}
\end{equation*}
This is the map $ \mathbb{D} \circ \mathbb{P}(a) $ and the claim follows.
\end{proof}

\begin{lemma}
\label{functoronb}
$ \mathbb{D} \circ \mathbb{P}(b)  = \mathbb{CK} \circ \mathbb{D}(b). $
\end{lemma}

\begin{proof}
The morphism $ b \colon P(2) \langle 1 \rangle \rightarrow P(1) $ gets mapped under Koszul duality to the morphism $ \mathbb{D}(b) \colon\quad $
\begin{equation*}
\xymatrix{
& 0 \ar[r] & P(1) \langle -1 \rangle \ar[r]^{a} \ar[d] & P(2) \langle -2 \rangle \ar[r]^{b} & P(1) \langle -3 \rangle \ar[r] & 0\\
0 \ar[r] & P(2) \ar[r]^{b} & P(1) \langle -1 \rangle \ar[r] & 0
}
\end{equation*}

Tensoring with the complex $ \mathbb{CK} $ gives rise to the morphism of bicomplexes indicated by the curved arrows in \eqref{bcomplex}, where for simplicity all the zero maps are omitted.

\begin{equation}
\label{bcomplex}
\xymatrix{
\cdots & \cdots  & \cdots
& \cdots & \cdots \\
P(2) \langle -6 \rangle \ar[u]^{-c} \ar[r] & 0 \ar[r] \ar[u] & P(2) \langle -8 \rangle \ar[u]^{-c}
& 0 \ar[r] \ar[u] &P(2) \langle -6 \rangle \ar[u]^{-c}   \\
P(2) \langle -4 \rangle \ar[u]^{c} \ar[r] & 0 \ar[r] \ar[u] & P(2) \langle -6 \rangle \ar[u]^{c}
& 0 \ar[r] \ar[u] &P(2) \langle -4 \rangle \ar[u]^{c}   \\
P(2) \langle -2 \rangle \ar[u]^{-c} \ar[r] \ar@/_1pc /[rrrr]_>{1} & 0 \ar[r] \ar[u] & P(2) \langle -4 \rangle \ar[u]^{-c}
& 0 \ar[r] \ar[u] &P(2) \langle -2 \rangle \ar[u]^{-c}  \\
P(1) \langle -1 \rangle \ar[u]^{b} \ar[r] \ar@/_1pc /[rrrr]_>{1} & 0 \ar[r] \ar[u] & P(1) \langle -3 \rangle \ar[u]^{b} \ar@/_1pc /[rruu]^<{a}
& 0 \ar[r] \ar[u] & P(1) \langle -1 \rangle \ar[u]^{b}  \\
0 \ar[u] & 0 \ar[u] & 0 \ar[u]
& 0 \ar[r] \ar[u] &0 \ar[u]
}
\end{equation}

This induces the following map on total complexes
$ \mathbb{CK} \circ \mathbb{D}(P(2) \langle 1 \rangle) \rightarrow \mathbb{CK} \circ \mathbb{D}(P(1)) $:

\begin{equation*}
\xymatrix{
0 \ar[r] & P(1) \langle -1 \rangle \ar[r]^{a} \ar[d]^{1} & P(2) \langle -2 \rangle \ar[r]^{b} \ar[d]^{1} & P(1) \langle -3 \rangle \ar[r] \ar[d]^{a} & 0 & \\
0 \ar[r] & P(1) \langle -1 \rangle \ar[r]^{a} & P(2) \langle -2 \rangle \ar[r]^{c} & P(2) \langle -4 \rangle \ar[r]^{c} & P(2) \langle -6 \rangle \ar[r]^{c} & \cdots
}
\end{equation*}

This is the map $ \mathbb{D} \circ \mathbb{P}(b) $ and the claim follows.

\end{proof}

\begin{theorem}
\label{sl2theorem}
There is an isomorphism of functors:
\begin{equation*}
\mathbb{D} \circ \mathbb{P} \cong \mathbb{CK} \circ \mathbb{D}.
\end{equation*}
\end{theorem}

\begin{proof}
This follows from Proposition ~\ref{functorsonprojectives} and Lemmas ~\ref{functoronc}, ~\ref{functorona}, and ~\ref{functoronb}.
\end{proof}

C.S.\\ Department of Mathematics, Endenicher Allee 60, 53115 Bonn  (Germany).\\ email: \email{stroppel@math.uni-bonn.de}

J.S.\\ Max Planck Institute for Mathematics, Vivatsgasse 7, 53111 Bonn (Germany). \\ email: \email{sussan@mpim-bonn.mpg.de}

\end{document}